%% file: maththesis.tex
\documentclass[11pt]{report}
\usepackage{mathrsfs, amsmath,amssymb,amsthm}
\input{macros}
\newtheorem{lemma}{Lemma}

\newtheorem{proposition}[lemma]{Proposition}
\newtheorem{corollary}[lemma]{Corollary}
\newtheorem{theorem*}{Theorem} 

\begin{document}
\title{Almost-Primes Represented by Quadratic Polynomials}
\author{Vishaal Kapoor\\B.\ Sc. (Mathematics) Simon Fraser University
\\
\\
  AN ESSAY SUBMITTED IN PARTIAL FULFILLMENT OF\\
  THE REQUIREMENTS FOR THE DEGREE OF\\
  Master of Science\\
  in\\
  THE FACULTY OF GRADUATE STUDIES\\
  (Department of Mathematics)\\
}
\date{April, 2006}
\maketitle

\begin{center}

\end{center}


\paragraph{Abstract}
In his paper {\it Almost-Primes Represented by Quadratic Polynomials}, Iwaniec
proved that the polynomial $n^2+1$ takes on values with at most two prime
factors (counted with multiplicity) infinitely often.  He states that ``in order
to avoid technical complications, we shall restrict our proof to the polynomial
$n^2+1.$''.  
In this exposition, we follow Iwaniec's proof and show that for any irreducible
quadratic polynomial $G(n)$ (satisfying some obviously necessary hypotheses),
$G(n)$ has at most two prime factors for infinitely many values of $n$.  

\tableofcontents

\paragraph{Acknowledgements}
First and foremost, I would like to thank my supervisor Greg Martin.  Without
his patience, guidance, and clear explanations, none of this would have been
possible.  I would also like to thank Erick Wong for always being available to
talk about mathematics at any hour of the day, and for teaching me about the
true power of the Schwarz.  And finally, I would like to thank my family for
their continued support, even though they still don't know quite what I do.

\nocite{*}
\input{chapter1}
\input{chapter2}
\input{chapter3}

\input{chapter4}
\input{chapter5}
\input{chapter6}
 \bibliographystyle{plain}
 \bibliography{maththesis}

\end{document}

%% file: macros.tex
\renewcommand{\mod}[1]{{\ifmmode\text{\rm\ (mod~$#1$)}\else\discretionary{}{}{\hbox{ }}\rm(mod~$#1$)\fi}}
\newcommand{\jacobi}[2]{\left(\displaystyle \frac{#1}{#2}\right)}
\newcommand{\cst}{\mathcal{G}_G}
\newcommand{\Cst}{\mathfrak{S}_G}
\newcommand{\pwr}{7/8}
\renewcommand{\nu}{\eta}
\renewcommand{\epsilon}{\varepsilon}

%% file: chapter1.tex
\setlength{\parskip}{0.5cm}
\chapter{Introduction}
\paragraph{} In 1978, Iwaniec proved in his paper {\it Almost-Primes Represented by Quadratic
Polynomials} \cite{almostprimes} that the polynomial $n^2+1$ takes on values with at most 2
prime factors (counted with multiplicity) infinitely often.  Such a result is an attempt to
generalize Dirichlet's theorem on primes in arithmetic progressions to higher
degree polynomials.  Heuristic arguments suggest that any irreducible
polynomial with integral coefficients, positive leading coefficient, and no fixed prime divisor takes on infinitely many prime values.  

In his paper, Iwaniec also states that it is possible, with technical complications, to prove that
\begin{theorem*}
For an irreducible polynomial $g(n) = an^2 + bn+c$ with $a>0$ and
odd $c$, there are infinitely many integers $n$ such that $g(n)$ has
at most 2 prime factors.  Moreover, if $x$ is sufficiently large, then 
\begin{eqnarray*}
|\{n \leq x; g(n) = P_2\}| > \frac{1}{77} \Gamma_g \frac{x}{\log x},
\end{eqnarray*}
where $\displaystyle \Gamma_g = \frac{1}{\deg g}\prod_{p} \left(1 -
\frac{\rho(p)}{p}\right) \left(1 - \frac{1}{p}\right)^{-1}$ and $\rho(p)$ is the
number of incongruent solutions of $g(n) \equiv 0 \mod p.$
\end{theorem*}
We say a number is almost-prime of order $r$ (denoted by $P_r$) if it has at most $r$ prime factors counted with
multiplicity.  In this exposition, we follow the techniques employed by Iwaniec to prove that
such an irreducible polynomial $g(n)=an^2+bn+c$ takes on $P_2$ values infinitely often.
Let $s = 4a(b^2-4ac).$  We note that as $g(n)$ has no fixed prime
divisors, there must be a residue class $t \mod s$ such that 
$g(t)$ is not congruent to $0 \mod p$ for any prime $p$ dividing $s.$  Let $G(n)$ be the polynomial
$g(s n + t),$ and let $\delta$ and $\Delta$ denote the discriminants of $g(n)$ and $G(n)$  respectively.  
To summarize, for the remainder of this exposition we have
\begin{align*}
\delta &= b^2 - 4ac,\\
s=\Delta &= 4a \delta,\ \textrm{and}\\
G(n) &= g(s n +t).
\end{align*}

%% file: chapter2.tex
\setlength{\parskip}{0.5cm}
\chapter{Richert's Weighted Sum}
\paragraph{}Define $\mathscr{A}$ to be the sequence of values of $f(n)$ for
$n\leq x,$ and let $\mathscr{A}_p$ denote the set of values in $\mathscr{A}$
that are divisible by $p.$  Ideally, we are interested in finding a lower bound for the number
of prime values in our sequence $\mathscr{A}$; but as this is difficult, we must settle for
$P_2$ values.  

\paragraph{}To detect these $P_2$ values, we make use of the weighted sieve.
Denote by $p_n$ and $\omega(n),$ the least prime factor of $n$ and the number of
prime divsisors of $n$ respectively.
Let
\begin{eqnarray*}
w_p(n) = \left\{ \begin{array}{ll}
\displaystyle 1-\frac{\log p}{\log x} & \textrm{if $p=p_n$},\\
\displaystyle
\frac{\log p_n}{\log x} & \textrm{if $p>p_n$ and $p<x^{1/2}$},\\
\displaystyle
1-\frac{\log p}{\log x} & \textrm{if $p>p_n$ and $p \geq x^{1/2}$}.
\end{array} \right. 
\end{eqnarray*}
Let $2 \leq \lambda < 3$ be a parameter, and define 
\begin{equation}
w(n) = 1 - \frac{1}{3-\lambda} \sum_{p \mid n, p<x} w_p(n).
\label{eqn_w}
\end{equation}
Then
\begin{lemma}
If $n \leq x^\lambda$ and $w(n) > 0$ then $n$ has at most 2 distinct
prime factors. \label{lemma_1}
\end{lemma}
\begin{proof}
We note that the sum in \eqref{eqn_w} consists of only positive terms that are $\leq 1$.  Hence if $n$
has two or more prime factors less than $x^{1/2}$, 
\begin{align*}
w(n) 
& \leq  
1 - \frac{1}{3-\lambda}\left\{1 -\frac{\log p_n}{\log x} + \frac{\log
p_n}{\log x} \right\}
 = \frac{2 - \lambda}{3 - \lambda}  \leq 0
\end{align*}
by our hypothesis on $\lambda.$  
Thus, we may assume that $n$ has at most one prime factor less than
$x^{1/2}$ and so we have
\begin{eqnarray*}
w(n) = 1 - \frac{1}{3-\lambda}
\sum_{\substack{p \mid n \\ p<x} }\frac{\log(x/p)}{\log x} 
\leq 1 - \frac{1}{3-\lambda}\left\{\omega(n) - \frac{\log n}{\log x}\right\} 
\leq \frac{3 - \omega(n)}{3 - \lambda}
\end{eqnarray*}
which is $\leq 0$ when $\omega(n)> 2.$  
\end{proof}

Let $\mathscr{A}$ be a sequence of positive integers $a \leq x^{\lambda}$ and
let $z \leq x^{1/2}.$  We note that Lemma \ref{lemma_1} gives us useful information
about only squarefree $P_2$s. 
\newcommand{\sumprime}{\mathop{\sum\nolimits'}}
\begin{lemma} \label{lemma_useful}
We have 
\begin{align*}
|\{a \in \mathscr{A} : \textrm{$a$ is a $P_2$}\}|  \geq W(\mathscr{A},z) +
O(xz^{-1/2}).
\end{align*}
\end{lemma}
\begin{proof}
For the non-squarefree numbers, we consider the set $|\{ G(n) :n \leq x,\
(G(n),P(z))=1, \textrm{~and~}  G(n) \equiv 0 \mod{p^2}\  \textrm{for some
prime}\ p\}| \leq |\{a \in \mathscr{A}: (a,P(z)) = 1),\ \textrm{$a$ is
non-squarefree}\} |.$ If $G(n)=an^2+bn+c = p^2 l$ for some $l$, since
$(G(n),P(z))=1,$ we must have $l\geq z$ (and $p\geq z$).  As $G(n) \leq Dn^2$
for some $D,$ we have $p\leq D xz^{-1/2}.$ Thus
\begin{displaymath}
\sum_{\substack{
	a \in \mathscr{A},\ (a,P(z)) = 1)\\ 
	\textrm{$a$ is non-squarefree}}} 1 
\leq
\sum_{z\leq p \leq D xz^{-1/2}}\ \sum_{n \leq x,\ G(n) \equiv 0 \mod{p^2}} 1,
\end{displaymath}
which is 
\begin{displaymath}
\leq \sum_{z\leq p \leq D xz^{-1/2}} 2 (xp^{-2}+1) \ll x z^{-1/2}.
\end{displaymath}
And for the squarefree numbers, Lemma \ref{lemma_1} gives
\begin{equation*}
|\{a \in \mathscr{A} : \textrm{$a$ is a $P_2$}\}| 
\geq 
\sumprime_{\substack{a \in \mathscr{A}\\ (a,P(z)) = 1}}  w(a),
\end{equation*}
where the summation is over squarefree numbers in $\mathscr{A}.$  Thus
\begin{align*}
W(\mathscr{A},z) &= \sum_{\substack{a\in \mathscr{A}\\ (a,P(z))=1}} w(a)  
 = 
\sumprime_{\substack{a \in \mathscr{A}\\ (a,P(z)) = 1}}  w(a) + O(xz^{-1/2});
\end{align*}
and so the conclusion follows.
\end{proof}
We will be concerned with the sequence $\mathscr{A} =
\{G(n):\ n\leq x\},$ so we fix $\lambda = 2+ \frac{D}{\log x}$ where $D$ some
constant such that $G(n) \leq Dn^2.$  As the main term we will obtain will be of a larger order of magnitude than $xz^{-1/2},$ it is sufficient to find a lower bound for
\begin{displaymath}
W(\mathscr{A},z) = \sum_{\substack{a\in \mathscr{A}\\ (a,P(z))=1}} w(a).
\end{displaymath}
\paragraph{}Write 
\begin{eqnarray}
W(\mathscr{A},z) & = & \sum_{\substack{a\in \mathscr{A}\\ (a,P(z))=1}}
\left(1- \frac{1}{3 - \lambda}\sum_{p \mid a, ~ p<x} w_p(a)\right) \\
		 & = & S(\mathscr{A},z) 
-  \frac{1}{3 - \lambda}\sum_{\substack{a\in \mathscr{A}\\(a,P(z))=1}} \sum_{p
\mid a, ~ p<x}
w_p(a). 
\end{eqnarray}
Interchanging order of summation gives,
\begin{displaymath}
W(\mathscr{A},z) = S(\mathscr{A},z) - 
\frac{1}{3 - \lambda} \sum_{ z \leq p < x} \sum_{\substack{a \in \mathscr{A}_p\\ (a,P(z)) = 1}} w_p(a).
\end{displaymath}
Considering the definition of $w_p(a),$ we divide the double sum into three cases
\begin{eqnarray*}
\left\{\begin{array}{l}
z \leq p < x^{1/2},\ \textrm{$p$ is the smallest prime factor of $a$}; \\
z \leq p < x^{1/2},\ \textrm{$p$ is not the smallest prime factor of $a;$ and},\\
x^{1/2}\leq p < x.
\end{array}\right\}
\end{eqnarray*}
This gives
\begin{align}
W(\mathscr{A},z) 
& =  S(\mathscr{A},z) - 
\frac{1}{3-\lambda}\bigg\{\sum_{z\leq p < x^{1/2}}\left(1-\frac{\log p}{\log x}\right) 
	S(\mathscr{A}_p,p) \nonumber \\&  - 
\sum_{z\leq p_1 < p < x^{1/2}}\frac{\log p_1}{\log x}
	S(\mathscr{A}_{p p_1}, p_1) 
-
\sum_{x^{1/2} \leq p < x}\left(1-\frac{\log p}{\log x}\right) 
		S(\mathscr{A}_p,z)\bigg\}.\label{eqn_W5}
\end{align}
Making use of the Buchstab formula
\begin{align*}
\sum_{z \leq p_1 < p} S(\mathscr{A}_{p p_1}, p_1) = S(\mathscr{A}_{p},z) -
S(\mathscr{A}_{p},p),\end{align*}
we add 
$\textstyle \frac{\log p}{\log x} \sum_{z \leq p_1 < p} S(\mathscr{A}_{p p_1},
p_1)$ to the middle sum of \eqref{eqn_W5} and subtract 
$\frac{\log p}{\log x}\left(S(\mathscr{A}_{p},z) - S(\mathscr{A}_p,p)\right)$
from the remaining sums to obtain
\begin{align}
W(\mathscr{A},z) 
&=
S(\mathscr{A},z) +
\frac{1}{3 - \lambda}\bigg\{\sum_{z \leq p < x^{1/2}} 
\sum_{z \leq p_1 < p} \frac{\log (p/p_1)}{\log x} S(\mathscr{A}_{p p_1},p_1)
\nonumber \\
&  -
\sum_{z \leq p < x^{1/2}}\bigg(\left(1-\frac{2\log p}{\log x}\right) 
	S(\mathscr{A}_p,p) + \frac{\log p}{\log x} S(\mathscr{A}_p,z)\bigg) \nonumber \\
&  -
\sum_{x^{1/2}\leq p < x}\left(1 - \frac{\log p}{\log x} \right)
		S(\mathscr{A}_p,z)\bigg\}.\label{eqn_W2}
\end{align}

%% file: chapter3.tex
\setlength{\parskip}{0.5cm}
\chapter{Linear Sieve with Error Term}
\label{ErrorTerm}
Let $\mathscr{B}$ be a finite sequence of $X$ integers.  We suppose the 
existence of a multiplicative function $\rho(d)$ that is used to approximate the
number of elements in $\mathscr{B}$ congruent to $0 \mod d.$   We also suppose
that $0\leq \rho(p) < p,$ for any prime $p.$  
Stated more precisely, 
\begin{align*}
&|\mathscr{B}_d| = |\{b\in\mathscr{B}; b\equiv 0 \mod d\}| \approx X
\frac{\rho(d)}{d},~~\textrm{and} \\
&0 \leq \rho(p) < p\ \textrm{ for any prime $p$}.
\end{align*}
We denote the error in our approximation by
\begin{eqnarray*}
r(\mathscr{B}; d) = |\mathscr{B}_d| - \frac{\rho(d)}{d}X.
\end{eqnarray*}
We also insist that there is a constant $K\geq 1$
such that for any $2\leq w < z,$ 
\begin{equation*}
\prod_{w\leq p < z} \left(1-\frac{\rho(p)}{p}\right)^{-1} \leq \frac{\log
z}{\log w} \left(1+\frac{K}{\log w}\right).
\end{equation*}
We also define 
$V(z) = \prod_{p<z}\left(1-\frac{\rho(p)}{p}\right).$

\paragraph{}Lastly, let $F(s)$ and $f(s)$ be the continuous solutions of the system
of differential-difference equations
\begin{eqnarray*}
\begin{array}{lll}
sf(s) & =0 & \textrm{if~} 0<s\leq 2,\\
sF(s) & =2e^C & \textrm{if~} 0<s\leq 3,\\
(sf(s))' & = F(s-1) & \textrm{if~} s>2,\\
(sF(s))' & = f(s-1) & \textrm{if~} s>3,
\end{array}
\end{eqnarray*}
where $C=0.577...$ is the Euler constant.  
\paragraph{} From \cite{newform}, we have
\begin{lemma} \label{lemma_2}
Let $z\geq 2,M\geq 2,$ and $N\geq 2.$  For any $\nu >0$ we have 
\begin{eqnarray*}
S(\mathscr{B},z) \leq V(z) X \{F(s) + E\} + 2^{\nu^{-7}}R(\mathscr{B};M,N),\\
S(\mathscr{B},z) \geq V(z) X \{f(s) - E\} - 2^{\nu^{-7}}R(\mathscr{B};M,N),
\end{eqnarray*}
where $s=\log MN/\log z,$ and $E\ll_{\epsilon} \nu s^2 e^{K} + \nu^{-8} e^{K-s} (\log MN)^{-1/3}.$
The error term $R(\mathscr{B}; M,N)$
is of the form
\begin{eqnarray*}
R(\mathscr{B}; M,N) = \sum_{\substack{m<M,n<N\\mn|P(z) }} a_m b_n
r(\mathscr{B};mn),
\end{eqnarray*}
where the coefficients $a_m,b_n$ are bounded by 1 in absolute value and depend at most on
$M,N,z,$ and $\nu.$
\end{lemma}

%% file: chapter4.tex
\setlength{\parskip}{0.5cm}
\chapter{Error Term}
\label{SectionErrorTerm}

In order to study the error term of the previous section, we look at a general
sum 
\begin{displaymath}
B(x; m,N) = \sum_{n<N, (n,m) = 1} b_n r(\mathscr{A}; mn).
\end{displaymath}
Here $\{b_n\}$ is a sequence of real numbers with $|b_n| \leq 1$ and $b_n=0$
when $n$ is not squarefree.  Note that we are considering the sequences $\{b_n\}$ supported
on squarefree $n$ because in the remainder term for our sieve, the sum is over
$m,n$ for $m,n$ dividing the squarefree number $P(z).$  Thus $a_m b_n$
will also be supported on squarefree numbers $n$ (and unimportantly $m$ as
well). 

For the remainder of this exposition, unless otherwise specified, the constant pertraining to Vinogradov's symbol $\ll$ is $\epsilon.$  

We will be interested in proving the following:
\begin{proposition} \label{proposition_1}
For $M<x$ and $\epsilon > 0,$ 
\begin{equation}
\sum_{M<m<2M} B^2(x; m,N) \ll (1+N^{15/4}M^{-9/4}x)x^{1+\epsilon}.
\label{dispersionequ}
\end{equation}
\end{proposition}
We defer the proof of \ref{proposition_1} until Section \ref{majorsucktion} as
several results will be required.
\begin{corollary} \label{corollary1}
Let $\epsilon > 0.$  Then
\begin{equation}
\sum_{m < x^{1-4\epsilon}} \Big| 
\sum_{\substack{n < x^{1/15 - \epsilon} \\ (n,m)=1}} b_n
r(\mathscr{A}; mn)\Big| \ll x^{1-\epsilon} \label{remaindersum}.
\end{equation}
\end{corollary}

\begin{proof}[Proof of Corollary.]  Let $N=x^{1/15-\epsilon}.$
Using dyadic blocks, we see it is sufficient to prove that 

\begin{displaymath}
\sum_{M<m<2M} |B(x; m,N)| \ll x^{1-3\epsilon/2}
\end{displaymath}
since this gives us that the sum in \eqref{remaindersum} is 
$ \ll  (1-4\epsilon) \cdot x^{1-3\epsilon/2} \cdot \log_2 x \ll  x^{1-3\epsilon/2}
x^{\epsilon/2} = x^{1-\epsilon}.$  For $M<x^{14/15-\epsilon},$ we may crudely bound 
\begin{eqnarray*}
|B(x; m, N)| & \leq & \sum_{n<N,~ (n,m)=1} |b_n| |r(\mathscr{A}; mn)|\\
& = & \sum_{n<N,~ (n,m)=1} |b_n| \Big| |\mathscr{A}_{mn}| - \frac{\rho(mn)}{mn}X
\Big|\\
& \leq & \sum_{n<N} \rho(mn) \ll \rho(m) N,
\end{eqnarray*}
as $\mathscr{A}$ is a polynomial sequence and $\rho$ is multiplicative.

For $x^{14/15-\epsilon} \leq M < x^{1-4\epsilon},$ we make use of the
Cauchy-Schwarz inequality 
\begin{eqnarray*}
\sum_{M<m<2M}B(x; m,N) \cdot 1 
&\leq &\Big\{ \sum_{M<m<2M} B(x;m,N)^2 \Big\}^{1/2} ~ \Big\{\sum_{M<m<2M}
1^2 \Big\}^{1/2}\\ 
& \leq & \left\{ (1+N^{15/4}M^{-9/4}x)x^{1+\epsilon} \right\}^{1/2}
M^{1/2}\\
& \leq & \left\{ 1+N^{15/4}M^{-9/4}x \right\}^{1/2} ~ x^{(1+\epsilon)/2}
x^{(1-4\epsilon)/2}\\
& \leq & \left\{ 1 + o_{\epsilon}(1) \right\}^{1/2} ~ x^{1-3\epsilon/2}
\ll x^{1-3\epsilon/2}
\end{eqnarray*}
This completes the proof.
\end{proof}

\paragraph{} The following result is based on the idea that over solutions $x
\mod{m}$ of $G(n) \equiv 0 \mod{m}$, the ratios $x/m$ are uniformly distributed
modulo 1.  We will require several lemmas for its proof so we defer the proof
until later in this section. 
\begin{proposition}
\label{lemma:4}
Let $q$ be a squarefree number with an odd divisor $d,$ $(d,\mu)=1,$ and
$\omega$ be a root of $G(x) \mod d.$  Let $M<M_1<2M$ and $0\leq \alpha < \beta <
1.$  Denote by $P(M_1,M; q,d,\mu,\omega,\alpha,\beta)$ the number of pairs of
integers $m,\Omega$ satisfying 
\begin{equation}
\begin{array}{lll}
M < m < M_1, & (m,q)=1, & m\equiv \mu \mod d,\\
\alpha m q \leq \Omega < \beta m q, & G(\omega) \equiv 0 \mod {mq},&
\Omega\equiv \omega \mod d. \label{Pconditions}
\end{array}
\end{equation}
Then for any $\epsilon>0$ we have
\begin{align*}
P(M_1, M; q,d,\mu,\omega,\alpha,\beta) &= \Cst (\beta-\alpha)(M_1 -M)
\rho\left(\frac{q}{d}\right)\frac{A(q)}{\phi(d)} + O((qM)^{\pwr+\epsilon}),
\end{align*}
where $A(q) = \frac{1}{(2,q)} (\phi(q)/q)^2,$ and $\Cst$ is given by the
product
\begin{align}
\Cst &:=
\frac{6}{\pi^2}\prod_{(p , 2a \delta q) > 1} \left(1-\frac{1}{p^2}\right) 
\prod_{\substack{p \mid \delta \\ (p,q)=1}} \left(1 - \frac{1}{p}\right) 
\prod_{\substack{p\mid 2a \\ (p,\delta q) = 1}} 
	\left(1 - \frac{\chi_{\delta}(p)+1}{p} + \frac{\chi_{\delta}(p)}{p^2}
	\right) \nonumber
	\\ & \qquad \cdot L(1,\chi_{\delta})\prod_{p \mid q}\left(1 -
\frac{\chi_{\delta}(p)}{p}\right).
\end{align}
Here $\chi_{\delta}$ is the Dirichlet character modulo $4\delta$ defined on
primes by $\chi_{\delta}(p) = \jacobi{\delta}{p}.$
\end{proposition}
\begin{corollary} \label{corollary_1}
Let $qq_1$ be a squarefree number with $\rho(qq_1) \neq 0.$  If $M<M_1<2M$ and
$0 \leq \alpha  < \beta < 1,$ then for any $\epsilon > 0,$
\begin{align*}
\sum_{\substack{M < m < M_1\\ (m,q q_1) = 1}}
\sum_{\substack{\alpha m q \leq \Omega < \beta m q\\ G(\Omega) \equiv 0 \mod
{mq}}} 1
& =  \Cst (\beta - \alpha) (M_1 - M) \rho(q) A(q q_1) \\
& \qquad + O((q q_1 M)^{\pwr+\epsilon}).
\end{align*}
\end{corollary}
\begin{proof} This follows immediately from Proposition
\ref{lemma:4} upon consideration of the quantity
\begin{align*}
\displaystyle \frac{1}{\rho(q_1)} P(M_1,M; 1,1,1,\alpha,\beta).
\end{align*}
\end{proof}
\begin{corollary} \label{corollary_2}
Let $q$ be a squarefree number with $\rho(q) \neq 0.$  If $M < M_1 < 2M,$ then
for any $\epsilon > 0$ 
\begin{eqnarray*}
\sum_{\substack{M < m < M_1\\ (m,q q_1) = 1}}
			1
& = & \Cst (M_1 - M) A(q) 
 + O((q M)^{\pwr+\epsilon}).
\end{eqnarray*}
\end{corollary}
\begin{proof} The sum considered is just $\frac{1}{\rho(q)}
P(M_1, M; q, 1, 1, 1, 0, 1).$
\end{proof}

\paragraph{}The following three lemmas will be used in the proof of Proposition \ref{lemma:4}
\begin{lemma} There exists a one-to-one correspondence between the solutions
$\Omega \mod D$ of $\Omega^2 + 1\equiv 0 \mod D$ and the pairs of integers
$(r,s)$ satisfying
\begin{eqnarray}
D = r^2 + s^2,\ (r,s)=1,\ |r|<s. \label{exponential_conditions}
\end{eqnarray}
The correspondence is obtained via the formula
\begin{eqnarray*}
\Omega = \frac{\overline{r}}{s}(r^2+s^2) - \frac r s,
\end{eqnarray*}
where $\overline{r}$ denotes the inverse of $r$ modulo $s.$  
\label{lemma_correspondence}
\end{lemma}
The following lemma about exponential sums is due to Hooley \cite{Hooley} and
follows from estimates for Kloosterman sums.
\begin{lemma} \label{lemma_6}
If $h$ and $s$ are integers and $0<r_2-r_1 < 2s,$ then 
\begin{eqnarray}
\sum_{\substack{r_1 < r < r_2,\ (r,s)=1 \\ r \equiv \lambda \mod {\Lambda}}}
	e \left(h \frac{\overline{r}}{s}\right) \ll s^{1/2 + \epsilon} (h,s)^{1/2}.
\end{eqnarray}
\end{lemma}
We will also need a smooth function to approximate the indicator function for
$[\alpha,\beta].$  
\begin{lemma}\label{smooth}
Let $2C < \beta - \alpha < 1-2C.$  Then there exist $A(t)$ and $B(t)$ such that
$$|\psi(t) - A(t)| = B(t)$$
and 
\begin{eqnarray*}
A(t) = \beta - \alpha + \sum_{h \neq 0} A_h e(ht)\\
B(t) = C + \sum_{h\neq 0} B_h e(ht)
\end{eqnarray*}
with Fourier coefficients $A_h,B_h$ satisfying $|A_h|,|B_h| \leq 
	\displaystyle \min\left(\frac{1}{|h|},\frac{{C}^{-2}}{|h|^3}\right) = C_h.$
\end{lemma}
The proof of Lemma \ref{smooth} is an easy exercise in Fourier analysis.  See
\cite{Fourier}.
\begin{proof}[Proof of Proposition \ref{lemma:4}.] Using Lemma 7 to replace the indicator function $\sum_{\alpha m q \leq \Omega < \beta mq}1$ in
\begin{equation*}
P(M_1,M; q,d,\mu,\omega,\alpha,\beta) 
= \sum_{\substack{M<m<M_1, ~(m,q)=1, ~m \equiv\mu \mod d\\ 0 \leq \Omega < qm,~
G(\Omega) \equiv 0 \mod {mq},~\Omega \equiv \omega \mod d }} \sum_{\alpha m q \leq
\Omega < \beta mq} 1,
\end{equation*} 
we obtain
\begin{align}
P(M_1,M; q,d,\mu,\omega,\alpha,\beta) 
= (\beta-\alpha)\sum_{\substack{M<m<M_1, ~(m,q)=1, ~m \equiv\mu \mod d\\ 0 \leq \Omega <
qm,~G(\Omega) \equiv 0 \mod {mq},~\Omega \equiv \omega \mod d }} 1 \nonumber
\\
+ O\bigg\{C M \rho(q) + \sum_{h \neq 0} C_h\bigg| \sum_{\substack{M<m<M_1, ~(m,q)=1, ~m \equiv\mu \mod d\\ 0 \leq \Omega < qm,~
G(\Omega) \equiv 0 \mod {mq},~\Omega \equiv \omega \mod d }}
e(h\Omega/mq)\bigg|\bigg\}. \label{eqn_P}
\end{align}
\paragraph{}For the main term, we have
\begin{align*}
\sum_{\substack{M<m<M_1, ~(m,q)=1,\ m \equiv\mu \mod d\\ 
				0 \leq \Omega < qm,\ G(\Omega) \equiv 0 \mod {mq},\ 
				\Omega \equiv \omega \mod d }} 1 
&= 
\sum_{\substack{M<m<M_1, ~(m,q)=1 \\m \equiv\mu \mod d}}\ 
	\sum_{\substack{0 \leq \Omega < qm,\ G(\Omega) \equiv 0 \mod {mq} \\
	\Omega \equiv \omega \mod d }} 1 \\
& = 
\sum_{\substack{M<m<M_1,\ (m,q)=1 \\
				m \equiv\mu \mod d}} \ \rho(qm/d)
\end{align*}
since $d|q,$ $(\mu,d)=1$ and $\omega$ is also a solution of $G$ modulo $d.$  As
$(d,mq/d)=1$ and $\rho$ is multiplicative, this sum is 
\begin{equation*}
\rho(q/d)  \sum_{\substack{M<m<M_1, ~(m,q)=1 \\m \equiv\mu \mod d}} \ \rho(m).
\end{equation*}

\paragraph{}To evaluate this sum, we more closely analyze $\rho(n).$  
By our construction of $G(n),$ we see that $G(n)$ has no solutions modulo $p$ for any
$p \mid \Delta$ (or more simply $p \mid 2a \delta$).  That is, for these primes $p,$ it
follows that $\rho(p^r)=0.$
\paragraph{}For $(p,2a\delta)=1,$ we may complete the square giving
\begin{eqnarray}
G(n)=an^2 + bn + c \equiv 0 \mod {p} & \Leftrightarrow & 4a^2 n^2 + 4ab^2n + 4ac
\equiv 0 \mod {p} \nonumber \\
&\Leftrightarrow & (2an+b)^2 \equiv \delta \mod {p}.
\label{eqn_completesquare}
\end{eqnarray}
Thus the number of solutions $\mod p$ is
$\rho(p)$ is $\chi_{\delta}(p)+1.$  And as $(p,2a\delta) = 1$, solutions lift uniquely to
$\mathbb{Z}_{p^r},$ so $\rho(p^r) = \chi_{\delta}(p) + 1.$ 

\paragraph{}As $0\leq \rho(n) \leq n,$ the Dirichlet series for $\rho$ converges
absolutely for $\textrm{Re } s>1.$  Thus for $\textrm{Re } s>1$ we have the product expansion:
\begin{align}
\sum_{n=1}^{\infty} \frac{\rho(n)}{n^s} 
 = 
\prod_{p}
\left\{ 
	\begin{array}{ll}
	1 + \frac{2}{p} + \frac{2}{p^2} + \frac{2}{p^3} + ... & 
		\textrm{if}\ (p,2a\delta) = 1,\ \textrm{and}\ \chi_{\delta}(p)=1,\\
	1 & \textrm{if}\ (p,2a\delta) = 1,\ \textrm{and}\ \chi_{\delta}(p)=-1,\\
	1 & \textrm{if}\ (p,2a\delta) > 1.
	\end{array}
\right\}
\label{EulerProduct}
\end{align}
Rewriting
\begin{eqnarray*}
1 + \frac{2}{p} + \frac{2}{p^2} + \frac{2}{p^3} + ... = 
	\left(1+\frac{1}{p^s}\right)\left(1-\frac{1}{p^s}\right)^{-1},
\end{eqnarray*}
and
\begin{equation*}
1 = \left(1+\frac{1}{p^s}\right)\left(1+\frac{1}{p^s}\right)^{-1},
\end{equation*}
our product \eqref{EulerProduct} becomes
\begin{align*}
\prod_{(p,2a\delta)=1}
	\left(1+\frac{1}{p^s}\right)\left(1 - \chi_{\delta}(p)
		\frac{1}{p^s}\right)^{-1} 
 = 
	\left(F(s)\frac{\zeta(s)}{\zeta(2s)}\right) L(s,\chi_{\delta}).
\end{align*}
Here 
\begin{align*}F(s) = \displaystyle \prod_{p \mid 2a\delta}
\left(1+\frac{1}{p^s}\right)^{-1}\left(1 - \chi_{\delta}(p)
\frac{1}{p^s}\right).\end{align*}

\paragraph{}Define $f(n)$ by
\begin{eqnarray*}
\sum_{n=1}^{\infty}\frac{f(n)}{n^s} 
& = &
F(s) \frac{\zeta(s)}{\zeta(2s)}.
\end{eqnarray*}
By considering the product expansion of $F(s)$ and $\zeta(s)/\zeta(2s)$, we may
deduce the values of $f$ on prime powers.  We divide the task into cases.

If $(p,2a\delta) = 1,$ then $f(p^r)$ can be obtained by considering the product
expansion of $\zeta(s)/\zeta(2s)$ since no factors arise in $F(s).$  
As $\zeta(s)/\zeta(2s) =\prod_{q} (1+q^{-s}),$ it is easy to see that $f(p^r)$
is 1 for $r \leq 1$ and 0 otherwise.

We now look at the cases involving $(p,2a\delta)>1.$  If $p|\delta$ then $\chi_{\delta}(p)=0$ and so the factor
$(1+p^{-s})^{-1}$ from $F(s)$ cancels the factor $(1+p^{s})$ from
$\zeta(s)/\zeta(2s).$ Thus $f(p^r)$ is 0 for $r>0$ (and 1 for $r=0$).  If
$(p,\delta)=1$ and $p|2a,$ then the factor in the product expansion corresponding to $p$ is
\begin{align*}
\bigg(1+\frac{1}{p^s}\bigg)^{-1}\bigg(1 - \chi_{\delta}(p)\frac{1}{p^s}\bigg)
\cdot \bigg(1+\frac{1}{p^s}\bigg)
= \left\{\begin{array}{ll}
\bigg(\displaystyle 1 - \frac{1}{p^s}\bigg) & \textrm{if } \chi_{\delta}(p)=1,\ \textrm{and} \\
\bigg(\displaystyle 1 + \frac{1}{p^s}\bigg) & \textrm{if } \chi_{\delta}(p)=-1.
\end{array}\right.
\end{align*}
Thus the sequence of values $(f(p^r))_{r=0}^{\infty}$ representing $f$ on prime
powers is $(1,-1,0,0,0,...)$ or $(1,1,0,0,0,...)$ as $\chi_{\delta}(p)=1$ or $-1$ respectively.

To summarize, we have
\begin{eqnarray*}
(f(p^r))_{r=0}^{\infty} &=& \left\{
\begin{array}{ll}
(1,1,0,0,0,...), & \textrm{if}\ (p,2a\delta)=1, \\
(1,0,0,0,...),	& \textrm{if}\ p \mid \delta, \\
(1,-1,0,0,0,...), & \textrm{if}\ (p,\delta) =1,\ p\mid 2a,\ \textrm{and}\ 
	\chi_{\delta}(p)=1,\\
(1,1,0,0,0,...), & \textrm{if}\ (p,\delta) =1,\ p \mid 2a,\ \textrm{and}\ 
	\chi_{\delta}(p)=-1.\\
\end{array}\right.
\end{eqnarray*}
Hitting $f$ with $\mu$ gives us $g$ with $f=1*g$:
\begin{eqnarray*}
(g(p^r))_{r=0}^{\infty} &=& \left\{
\begin{array}{ll}
(1,0,-1,0,0,0,...), & \textrm{if}\ (p,2a\delta)=1, \\
(1,-1,0,0,0,...),	& \textrm{if}\ p \mid \delta, \\
(1,-2,1,0,0,...), & \textrm{if}\ (p,\delta) =1,\ p\mid 2a,\ \textrm{and}\ 
	\chi_{\delta}(p)=1,\\
(1,0,-1,0,0,...), & \textrm{if}\ (p,\delta) =1,\ p \mid 2a,\ \textrm{and}\ 
	\chi_{\delta}(p)=-1.\\
\end{array}\right.
\end{eqnarray*}

\paragraph{} To evaluate $\rho = \chi_{\delta}*(1*g)$ we will make use of
Dirichlet's hyperbola method twice.  
The first application gives
\begin{eqnarray}\label{eqn_rho1}
\rho(q/d) \sum_{\substack{M<m<M_1,\ (m,q)=1 \\ m \equiv\mu \mod d}}\ \rho(m) =
\rho(q/d) 
	\sum_{\substack{a \leq M^{1/2} \\ (a,q)=1}} \chi_{\delta}(a) \
	\sum_{\substack{\frac M a < b < \frac{M_1}a,\ (b,\ q/d)=1 \\ 
					b \equiv \mu \overline{a} \mod d}} f(b)  \\ +\  
\rho(q/d) 
	\sum_{\substack{ b < 2 M^{1/2} \\ (b,q)=1}} \ 
	\sum_{\substack{\max(\frac M b,M^{1/2}) < a < \frac {M_1} b \\ 
					a \equiv \mu\overline{b} \mod d,\ (a,q/d)=1}} 
	\chi_{\delta}(a).\nonumber
\end{eqnarray}

The following lemma will help us with the error terms that occur in the
estimation of $\sum f(b).$
\begin{lemma} 
\label{lemma_bound}
We have
\begin{align*}
\sum_{A<i<B} |g(i)| \ll B^{1/2} \qquad \textrm{and} \qquad \sum_{A\leq i}
\frac{|g(i)|}{i} \ll \frac 1 A.
\end{align*}
\end{lemma}
\begin{proof}
We note that any $i$ can be written as $uv$ with $(v,2a\delta)=1$ and $u$
satisfying the property that each of its prime divisors also divide $2a\delta.$  We also note that if $p \mid 2a\delta$ then $g(p^r)=0$ for $r>2.$   Thus
\begin{align*}
\sum_{A < i < B}  |g(i)|
= \sum_{u: p|u \Rightarrow p \mid 2a\delta} |g(u)| \sum_{\substack{A/u < v < B/u
\\ (v,2a \delta)=1}} |g(v)|
\leq \sum_{u \mid (2a\delta)^2} |g(u)| \sum_{\substack{A/u < v < B/u
\\ (v,2a \delta)=1}} |g(v)|.
\end{align*}
By definition, $g(v)$ is 0 if $v$ is not a square, and $|g(v)| \leq 1$
otherwise.  Thus
\begin{align*}
\sum_{A < i < B}  |g(i)|
\leq \sum_{u \mid (2a\delta)^2} g(u) (B/u)^{1/2} \ll B^{1/2}.
\end{align*}
Similarily for $N>A,$
\begin{align*}
\sum_{A\leq i < N}\frac{|g(i)|}{i}
& \leq \sum_{ u | (2a\delta)^2}\frac{|g(u)|}{u}  \sum_{\substack{A/u \leq v < N/u \\
(v,2a\delta)=1}}\frac{|g(v)|}{v} 
 \leq \sum_{ u | (2a\delta)^2}\frac{|g(u)|}{u}  \sum_{\substack{A/u \leq v <
\infty \\
(v,2a\delta)=1}}\frac{|g(v)|}{v}.
\end{align*}
By the same reasoning above, 
\begin{align*}
 \sum_{\substack{A/u \leq v < \infty \\ (v,2a\delta)=1}}\frac{|g(v)|}{v} \leq
 \frac{u}{A},
\end{align*}
and so taking the limit as $N$ tends to $\infty$ gives
\begin{align*}
\sum_{A\leq i < \infty}\frac{|g(i)|}{i}
\leq \frac{1}{A}\sum_{ u | (2a\delta)^2}|g(u)| \ll \frac{1}{A}.
\end{align*}
\end{proof}
As $\sum_{(i,q)=1} g(i)/i$ is absolutely convergent by Lemma \ref{lemma_bound}, it has a product
expansion which we denote as
\begin{align}
\cst &:=
\frac{6}{\pi^2}\prod_{(p , 2a \delta q) > 1} \left(1-\frac{1}{p^2}\right) \cdot
\prod_{\substack{p \mid \delta \\ (p,q)=1}} \left(1 - \frac{1}{p}\right)
\nonumber\\
&\qquad \cdot \prod_{\substack{p\mid 2a \\ (p,\delta q) = 1}} 
	\left(1 - \frac{\chi_{\delta}(p)+1}{p} + \frac{\chi_{\delta}(p)}{p^2}
	\right).
\end{align}

\paragraph{}Continuing with \eqref{eqn_rho1}, we look at
\begin{align}\label{eqn_sum_f}
\sum_{\substack{M/a < b < M_1/a,\ (b, q/d)=1 \\ 
			b \equiv \mu \overline{a} \mod d}} f(b) 
& =
\sum_{\substack{M/a < ij < M_1/a,\ (ij,q/d)=1 \\ 
				ij \equiv \mu \overline{a} \mod d}} g(i) \cdot 1(j) .
\end{align}
Applying Dirichlet's hyperbola method gives
\begin{align} \label{eqn_hyp2}
\sum_{\substack{M/a < b < M_1/a,\ (b, q/d)=1 \\ 
			b \equiv \mu \overline{a} \mod d}} f(b)  & = 
\sum_{i < (M_1/a)^{1/2},\ (i, q)=1} g(i) \ 
\sum_{\substack{M/ai < j < M_1/ai\\ j \equiv \mu \overline{a
i} \mod d,\ (j, q/d)=1}} 1 \ \nonumber \\ &\qquad +
\sum_{\substack{ j < 2 (M/a)^{1/2} \\ (j,q)=1}} \ 
\sum_{\substack{\max(M/aj,\ (M/a)^{1/2})
				< i < M_1/aj \\ 
				i \equiv \mu\overline{a j} \mod d,\ (i,q/d)=1}} g(i),
\end{align}
the latter sum to be collected in our error.  The first sum is
\begin{eqnarray*}
\sum_{i < (M_1/a)^{1/2},\ (i, q)=1} g(i) \ 
\left\{\phi\left(\frac{q}{d}\right) \frac{M_1-M}{a i q } + O(\tau(q))\right\}.
\end{eqnarray*}
If we add the tail of the series and collect error terms, we have
\begin{align*}
\phi\left(\frac{q}{d}\right) \frac{M_1 - M}{aq} \mathcal{G}_G & + 
O\bigg\{
	\sum_{M/a < i < M_1/a} |g(i)| \tau(q) \ \\ &\qquad+
	\phi\left(\frac{q}{d}\right) \frac{M_1 - M}{aq} 
	\sum_{i \geq (M_1/a)^{1/2},\ (i, q)=1} \frac{|g(i)|}{i} \ 
	\bigg\}, 
\end{align*}
which by Lemma \ref{lemma_bound} is just
\begin{eqnarray}\label{eqn_sumf1}
\phi\left(\frac{q}{d}\right) \frac{M_1 - M}{aq} \mathcal{G}_G + 
O\Big\{\left(\frac{M_1}{a}\right)^{1/2} \tau(q)\Big\}.
\end{eqnarray}
For the second sum in \eqref{eqn_hyp2}, we make use of Lemma \ref{lemma_bound}
again to find that it is
\begin{eqnarray} \label{eqn_sumf2}
\ll \sum_{\substack{ j < 2 (M/a)^{1/2} \\ (j,q)=1}} 
	\left(\frac{M}{a j}\right)^{1/2} 
\ll 
 \left(\frac{M}{a }\right)^{3/4}.
\end{eqnarray}
Collecting \eqref{eqn_sumf1} and \eqref{eqn_sumf2} gives 
\begin{eqnarray}\label{sumf_final}
\sum_{\substack{M/a < b < M_1/a,\ (b, q/d)=1 \\ 
			b \equiv \mu \overline{a} \mod d}} f(b) 
= 
\phi\left(\frac{q}{d}\right) \frac{M_1 - M}{aq} \mathcal{G}_G  
+ O\Big(\tau(q) \left(\frac{M}{a }\right)^{3/4}\Big).
\end{eqnarray}
\paragraph{}Returning to \eqref{eqn_rho1}, we see that 
\begin{eqnarray*}
\rho(q/d) \sum_{\substack{M<m<M_1,\ (m,q)=1 \\ m \equiv\mu \mod d}}\ \rho(m)
= \rho(q/d) \cst \phi\left(\frac{q}{d}\right) \frac{M_1 - M}{q} 
	\sum_{\substack{a \leq M^{1/2} \\ (a,q)=1}} \chi_{\delta}(a)/a  
	\\
\qquad +\ O\Big\{\rho(q) M^{3/4}
		\sum_{\substack{a \leq M^{1/2} \\ (a,q)=1}}
			\frac{|\chi_{\delta}(a)|}{a^{3/4}} + 
 \rho(q/d) 
	\sum_{\substack{ b < 2 M^{1/2} \\ (b,q)=1}} \ 
	\sum_{\substack{\max(M/b,M^{1/2}) < a < M_1/b\\ 
					a \equiv \mu\overline{b} \mod d,\ (a,q/d)=1}} 
	|\chi_{\delta}(a)|\Big\}.
\end{eqnarray*}
The error term simplifies to $O(M^{7/8}),$ and as $\sum_{a \geq M^{1/2},\
(a,q)=1} \chi_{\delta}/a \ll M^{1/2}$ we have
\begin{eqnarray}
\rho(q/d) \sum_{\substack{M<m<M_1,\ (m,q)=1 \\ m \equiv\mu \mod d}}\ \rho(m)
= 
\rho(q/d) \Cst \phi\left(\frac{q}{d}\right) \frac{M_1 - M}{q} 
+ O\Big(\rho(q) M^{7/8}\Big),
\label{eqn_rho}
\end{eqnarray}
where 
\begin{eqnarray*}
\Cst = \cst \cdot L(1,\chi_{\delta})\prod_{p \mid q}\left(1 -
\frac{\chi_{\delta}(p)}{p}\right).
\end{eqnarray*}


For the exponential sum in \eqref{eqn_P}, we first eliminate the condition $(m,q)=1$ using
the principle of inclusion-exclusion as follows:
\begin{equation}
\sum_{m,\Omega} e(h\Omega/D)
= \sum_{l \mid q/d} \mu(l)
\sum_{\substack{qM<d<qM_1,~ D\equiv 0\mod {lq},~D\equiv\mu q \mod {dq}\\
0\leq \Omega < D, ~G(\Omega)\equiv 0 \mod D,~ \Omega \equiv \omega \ mod d}}
e(h\Omega/D). \label{eqn_exp_sum}
\end{equation}

Now we would like to make use of the bijection from lemma
\ref{lemma_correspondence} to parameterize the solutions in the above equation.
Since the inner sum runs only over $\Omega$ such that $G(\Omega) \equiv 0 \mod
D,$ we must have $(D,2a\delta) = 1.$  Hence we may complete the square as in
\eqref{eqn_completesquare} giving 
\begin{eqnarray*}
G(\Omega) \equiv 0 \mod D & \Leftrightarrow & 
		(2a\Omega + b) ^2 \equiv \delta \mod D \\
& \Leftrightarrow &
\Omega = \overline{2a}\left(\frac{\overline{r}}{s} (r^2+s^2) - \frac{r}{s}
-b\right) \\
& & \textrm{and $(r,s)$ satisfy \eqref{exponential_conditions}}
\end{eqnarray*}
(here $\overline{2a}$ denotes the inverse of $2a$ modulo $D$ and as usual $\overline{r}$ denotes the inverse of $r$ modulo $s$).
Writing $D=r^2+s^2$ with $(r,s)=1$, and $|r|<s,$ the conditions 
\begin{eqnarray*}
D \equiv 0 \mod {lq},\ r^2+s^2 \equiv \mu q \mod {dq},\ \textrm{and}\ 
	r+\omega s \equiv 0 \mod d,
\end{eqnarray*}
are equivalent to 
\begin{eqnarray*}
r^2+s^2 \equiv 0 \mod {lq},\ r^2+s^2 \equiv \mu q \mod {dq},\  \textrm{and}\ 
	r+\omega s \equiv 0 \mod d.
\end{eqnarray*}
Thus, equation \eqref{eqn_exp_sum} is
\begin{align*}
\sum_{D,\Omega} e(h\Omega/D) &\leq \tau(q) 
\sum_{(qM/2)^{1/2} < s < (2qM)^{1/2}} 
\sup_{\substack{r_1,r_2 \\ \lambda, \Lambda}}
\bigg| \sum_{\substack{(r,s) = 1,\ r_1 < r < r_2 \\ r \equiv \lambda 
					\mod \Lambda}}
				e\left(
					h \overline{2a}\left(\frac{\overline{r}}{s} - \frac{r}{s(r^2+s^2)}
						-b\right)
				\right)\bigg| \\
& = 
\tau(q)
\sum_{( qM/2)^{1/2} < s < (2qM)^{1/2}} 
\sup_{\substack{r_1,r_2 \\ \lambda, \Lambda}}
\bigg| \sum_{\substack{(r,s) = 1,\ r_1 < r < r_2 \\ r \equiv \lambda 
					\mod \Lambda}}
				e\left(
					h \overline{2a} \frac{\overline{r}}{s}\right) 
				e \left(
					-h \frac{\overline{2a}r}{s(s^2+r^2)}\right)\bigg|.
\end{align*}
Here the supremums are over integers $r_1,r_2, \lambda, \Lambda$ satisfying the
constraint $0 < r_2 - r_1 < 2s.$
To bound the inner sum, we use partial summation.  Define 
\begin{eqnarray*}
E(t) = 
\sum_{\substack{(r,s) = 1,\ r_1 < r < t \\ r \equiv \lambda 
					\mod \Lambda}}
				e\left(
					h \overline{2a} \frac{\overline{r}}{s}\right).
\end{eqnarray*}
Then 
\begin{align*}
\sum_{\substack{(r,s) = 1,\ r_1 < r < r_2 \\ r \equiv \lambda 
					\mod \Lambda}}
				e\left(
					h \overline{2a} \frac{\overline{r}}{s}\right) &
				e \left(
					-h \frac{\overline{2a}r}{s(s^2+r^2)}\right) \\
& = 
\int_{r_1}^{r_2}  e \left(-h \frac{\overline{2a}t}{s(s^2+t^2)}\right)
	\ dE(t) \\
& = 
E(t) e \left(-h \frac{\overline{2a}t}{s(s^2+t^2)}\right)\bigg|^{r_2}_{r_1} 
-\ \int_{r_1}^{r_2} E(t) \ d\left(e \left(-h
	\frac{\overline{2a}t}{s(t^2+s^2)}\right)\right) \\
& \ll 
s^{1/2+\epsilon} (h,s)^{1/2}\left(1 + \int_{r_1}^{r_2} \frac{h
|t^2-s^2|}{s(s^2+t^2)^2} \ dt\right),
\end{align*}
the last line following from Lemma \ref{lemma_6}.  The triangle inequality gives
$|t^2-s^2| \leq s^2+t^2$ so 
\begin{align*}
\sum_{\substack{(r,s) = 1,\ r_1 < r < r_2 \\ r \equiv \lambda 
					\mod \Lambda}}
				e\left(
					h \overline{2a} \frac{\overline{r}}{s}\right) 
				e \left(
					-h \frac{\overline{2a}r}{s(s^2+r^2)}\right) 
& \ll  
s^{1/2+\epsilon} (h,s)^{1/2}\left(1 + \frac{h}{s^2}\int_{r_1}^{r_2} 
	\frac{s}{(t/s)^2+1} \ dt\right) \\
&\ll 
s^{1/2+\epsilon} (h,s)^{1/2}\left(1 + \frac{h}{s^2}\right).
\end{align*}	
Consequently, 
\begin{eqnarray*}
\sum_{D,\Omega} e(h \Omega/D) & \ll &
\tau(q) (qM)^{1/4+\epsilon} \left(1 + \frac{h}{qM}\right) 
	\sum_{s < (2qM)^{1/2}} (h,s)^{1/2} \\
& \ll &
\tau(h) (qM)^{3/4+\epsilon} \left(1 + \frac{h}{qM}\right).
\end{eqnarray*}

\paragraph{} Recall from Lemma \ref{smooth} that $ C_h \leq \min
(|h|^{-1},C^{-2} |h|^{-3}).$  Hence, the error
term in \eqref{eqn_P} is 
\begin{align*}
\ll CM \rho(q) + \sum_{h\neq 0} C_h \left(1+\frac{h}{qM}\right)(qM)^{3/4+\epsilon} \tau(h) 
\ll 
\left(1+\frac{1}{C qM}\right) \log^2 C.  
\end{align*}
Choosing $C=1/qM$ gives an error term 
$\ll (qM)^{3/4+\epsilon}$
which completes the proof of Proposition \ref{lemma:4}.
\end{proof}
\section{Proof of Proposition \ref{proposition_1}}\label{majorsucktion}
By definition of $B(x; m,N),$
\begin{equation}
B(x; m,N) = \sum_{\substack{n<N \\ (n,m)=1}} b_n |\mathscr{A}_{mn}| - \frac{\rho(m)}{m}x
\sum_{\substack{n<N \\ (n,m)=1}} b_n \frac{\rho(n)}{n}. \label{defnB}
\end{equation}
We may write the first sum as
\begin{eqnarray*}
\sum_{n<N,~(n,m)=1} b_n |\mathscr{A}_{mn}|
&=&\sum_{\substack{n<N \\ (n,m)=1}} b_n \sum_{\substack{k<x \\ G(k) \equiv 0
\mod{mn}} } 1\\
&=&
\sum_{\substack{0<v<m \\ G(v) \equiv 0 \mod m}} \sum_{\substack{n<N \\ (n,m)=1}}
b_n
\sum_{\substack{k<x,~k\equiv v \mod m \\ G(k) \equiv 0 \mod n}} 1.
\end{eqnarray*}
And the approximation in \eqref{defnB} can be re-written as
\begin{eqnarray*}
\frac{\rho(m)}{m}~x~\sum_{\substack{n<N\\(n,m)=1}} b_n \frac{\rho(n)}{n} 
&=&
\sum_{\substack{0<v<m \\ G(v) \equiv 0 \mod m}} \sum_{\substack{n<N \\ (n,m)=1}}
Y(m),
\end{eqnarray*}
where 
\begin{displaymath}
Y(m) =  \sum_{\substack{n<N \\ (n,m)=1}} b_n \frac{\rho(n)}{n}.
\end{displaymath}
Thus \eqref{defnB} is
\begin{equation}
B(x; m,N) =
\sum_{\substack{0<v<m \\ G(v) \equiv 0 \mod m}} \Big\{ \sum_{\substack{n<N \\ (n,m)=1}}
b_n \sum_{\substack{k<x \\ k\equiv v \mod m \\ G(k) \equiv 0 \mod n}} 1 -
\frac{x}{m}~ Y(m) \Big\}.
\end{equation}
For sake of applying Cauchy-Schwarz again, we consider
\begin{eqnarray*}
\mathscr{M}(x; M,N) &=&
\sum_{M<m<2M}\sum_{\substack{0<v<m \\ G(v) \equiv 0 \mod m}} 
\Big\{\sum_{\substack{n<N \\ (n,m)=1}} b_n \sum_{\substack{k<x,~k\equiv v \mod m
\\ G(k) \equiv 0 \mod n}} 1 - \frac{x}{m} ~ Y(m) \Big\}^2\\
&=& 
\sum_{M<m<2M}\sum_{\substack{0<v<m \\ G(v) \equiv 0 \mod m}}
\Big\{ \sum_{\substack{n_1,n_2<N\\  (n_1 n_2,m)=1}} b_{n_1}b_{n_2}
\sum_{\substack{k_1,k_2<x \\ k_1,k_2\equiv v \mod m \\ G(k_1),G(k_2) \equiv 0 \mod
n}} 1 - 2\frac{x}{m} ~  \\& &  Y(m) \sum_{\substack{n<N \\ (n,m)=1}}b_n
\sum_{\substack{k<x,k\equiv v\mod m \\ G(k)\equiv0 \mod n}} 1 + \left(\frac x m
Y(m)\right)^2 \Big\}^2\\
&=&W(x; M,N) - 2xV(x; M,N) + x^2 U(M,N),
\end{eqnarray*}
taking $W,V,U$ to be the respective quantities.  
Over the next three subsections, we will obtain sufficient bounds for $U,V$ and
$W.$
\subsection{Estimation of $U(M,N)$}
We have
\begin{eqnarray*}
U(M,N) &=& \sum_{M<m<2M} \sum_{\substack{0<v<m \\ G(v) \equiv 0 \mod
m}}\left(\frac{1}{m}\sum_{n<N,~(n,m)=1}\frac{\rho(n)}{n}\right)^2\\
&=&\sum_{M<m<2M} \sum_{\substack{0<v<m \\ G(v) \equiv 0 \mod
m}}\frac{1}{m^2}\sum_{\substack{n_1,n_2 <N\\(n_1n_2,m)=1}}
\frac{\rho(n_1)\rho(n_2)}{n_1 n_2}\\
&=&\sum_{n_1,n_2<N} b_{n_1} b_{n_2} \frac{\rho(n_1) \rho(n_2)}{n_1 n_2}
\sum_{\substack{M<m<2M \\ (m,n_1 n_2)=1}} \frac{\rho(m)}{m^2}.
\end{eqnarray*}

Let $F(t) = \sum_{\substack{M<m<t\\(m,n_1n_2)=1}}1.$  Summing by parts and using Corollary \ref{corollary_2} gives
\begin{align}
\sum_{\substack{M<m<2M \\ (m,n_1 n_2)=1}}  \frac{\rho(m)}{m^2} &=  
\int_{M}^{2M^{-}}\frac{1}{t^2} \ dF(t) \nonumber \\
&= \frac{1}{4}M^{-2}\sum_{\substack{M<m<2M\\(m,n_1n_2)=1}} \rho(m) +
2\int_{M}^{2M}\sum_{\substack{M<m<t\\(m,n_1n_2)=1}} \rho(m)\frac{1}{t^3} ~~dt
\nonumber \\
&= \frac{\Cst}{4}M^{-1} A(n_1 n_2) + 2 \Cst A(n_1 n_2) \int_{M}^{2M}
\frac{(t-M)}{t^3} ~dt \nonumber \\
&\qquad + O\left\{\left(\frac{1}{4}M^{-2}+\int_{M}^{2M} \frac{1}{t^3} dt\right)(n_1 n_2
M)^{\pwr+\epsilon}\right\}\nonumber \\
&= \frac{\Cst}{4} M^{-1}A(n_1 n_2) + 2\Cst A(n_1 n_2)
\left(\frac{-1}{t}\rvert_M^{2M} -
M\frac{-2}{t^2}\arrowvert_M^{2M}\right)\nonumber \\
& \qquad + O\left(M^{-2}(n_1 n_2 M)^{\pwr+\epsilon}\right)\nonumber \\
&= \frac{\Cst}{2} A(n_1 n_2)M^{-1} +O((n_1 n_2
M)^{\pwr+\epsilon}M^{-2}).\label{eU}
\end{align}
This is sufficient for our purposes.  

\subsection{Estimation of $V(x; M,N)$}
We have by definition of $V(x; M,N)$

\begin{eqnarray*}
V(x; M, N) & = & \sum_{\substack{M<m<2M\\0\leq v < m, ~G(v) \equiv 0 \mod m}} \frac{1}{m}
\bigg(\sum_{\substack{n_1<N\\(m,n_1)=1}} b_{n_1} \sum_{\substack{k<x, G(k)\equiv
0\mod m \\ k\equiv v \mod m}} 1\bigg)\cdot
\bigg(\sum_{\substack{n_2<N\\(n_2,m)=1}} b_{n_2} \frac{\rho(n_2)}{n_2}\bigg).
\end{eqnarray*}
Let the symbol $\Theta_{n_1,n_2}$ denote the set of triples $(m,v,k)$ satisfying
\begin{equation}
\begin{array}{lll}
M<m<2M, & (m,n_1 n_2) = 1;\\
0\leq v < m,& G(v) \equiv 0 \mod m;\\
k<x, &k\equiv v \mod m,& G(k)\equiv 0 \mod {n_1},
\end{array}\label{ThetaEqn}
\end{equation}
and let 
\begin{align*}
S(n_1,n_2; x,M) = \sum_{\Theta_{n_1,n_2}} \frac{1}{m}.
\end{align*}
Then 
\begin{eqnarray*}
V(x;M,N) = \sum_{n_1,n_2<N} b_{n_1} b_{n_2} \frac{\rho(n_2)}{n_2}
S(n_1,n_2,x,M).
\end{eqnarray*}
Writing $k=v+ml$ for $l<x/M,$ we may replace the conditions on $k$ in
\eqref{ThetaEqn} giving

\begin{equation}
\begin{array}{lll}
M<m<\textrm{min}(2M,(x-v)/l), & (m,n_1 n_2) = 1;\\
0\leq v < m,& G(v) \equiv 0 \mod m,&G(ml+v) \equiv 0 \mod {n_1}.
\end{array}\label{Theta2Eqn}
\end{equation}

Note that $M<m=(k-v)/l<(x-v)/l$ implies that $l<(x-v)/M\leq x/m$ so
\eqref{ThetaEqn} and \eqref{Theta2Eqn} are equivalent.  As $0\leq v < m,$
replacing $(x-v)/l$ by $x/l$ gives (small) error:
\begin{align*}
\sum_{1\leq l < x/M}\sum_{(x-2M)/l<m<x/l} 
m^{-1} \sum_{\substack{0\leq v < m \\ G(v) \equiv 0 \mod {m}}}
								\sum_{n_1 \mid G(ml + v)} 1
\end{align*}
which is 
\begin{align*}
\ll  \sum_{1\leq l < x/M}\sum_{(x-2M)/l<m<x/l}
m^{-1} \rho(m) 
&\ll  \displaystyle \sum_{1\leq l < x/M
}\sum_{(x-2M)/l<m<x/l} m^{-1} x^{\epsilon/2} \\
&\ll  \displaystyle \frac{x}{M} \cdot \left(\frac{2M}{l} \cdot
\frac{l}{x}\right) x^{\epsilon/2} \ll x^{\epsilon}.
\end{align*}

\paragraph{} In an attempt to further simplify matters, we let $c$ be $l \mod
{n_1}$ ($0\leq c < n_1$).  This allows us to replace the conditions
$G(v) \equiv 0 \mod m$ and $G(ml+v) \equiv 0 \mod {n_1}$ with $G(\theta) \equiv 0
\mod {m n_1}$ where $\theta = cm+v.$  Note that $m$ and $n_1$ are relatively prime
by assumption. Also as $0\leq v < m$ we have the condition $cm\leq \theta <
(c+1)m.$  Let $S^{*}(n_1,n_2; x,M)$ denote the sum approximating $S(n_1,n_2;x,
M)$.  Then 

\begin{eqnarray}
S^*(n_1,n_2; x,M) = \sum_{l<x/M}
\sum_{\substack{M<m<\textrm{min}(2M,x/l), (m,n_1,n_2) = 1\\
cm \leq \Omega < (c+1) m, G(\Omega ) \equiv 0 \mod {m n_1} }} m^{-1}.
\label{eqn1}
\end{eqnarray}

Define 
\begin{equation*}
S_0(t) = \sum_{\substack{M<m<t,\ (m,q q_1) = 1 \\
				\alpha m q \leq \Omega < \beta m q,\ G(\Omega)\equiv 0 \mod {mq}}} 
					m^{-1}.
\end{equation*}
Then from using partial summation with Corollary \ref{corollary_2}, we have

\begin{align*}
S_0(M_1)  & = 
\int_{M}^{M'} S(t)\frac{1}{t^2}\ dt + S(t)\frac{1}{t}|_{M}^{M'} \\
& = 
\int_{M}^{M'} \left\{\Cst (\beta - \alpha)(t-M) \rho(q) A(q q_1) 
  + O((q q_1 M)^{\pwr + \epsilon})\right\}\frac{1}{t^2}\ dt \\
& \qquad + O\left(\frac{1}{M'}\big\{\rho(q) A(q q_1) (M'-M) + (q q_1
M)^{\pwr+\epsilon}\big\} \right)  \\
& = 
\int_{M}^{M'} \Cst (\beta - \alpha)\frac{1}{t} \rho(q) A(q q_1)
			\ dt + O((q q_1 M)^{\pwr+\epsilon}) \\
& = 
\Cst (\beta - \alpha) \log\left(\frac{M'}{M}\right) \rho(q) A(q q_1)
	+ O((q q_1 M )^{\pwr+\epsilon}) 
\end{align*}
for any $0\leq \alpha < \beta \leq 1$ and $\epsilon >0.$  
Applying this to \eqref{eqn1} with $\alpha=c/a$ and
$\beta=(c+1)/q$ gives 

\begin{align*}
S^*(n_1,n_2; x,M) & =
\sum_{l<x/M}
\Cst  \log\left(\textrm{min}\left(2,\frac{x}{lM}\right)\right)
\frac{\rho(n_1)}{n_1}
A(n_1 n_2) + O(M^{-1} (n_1 n_2 M)^{\pwr+\epsilon})\\
& = 
\frac{\Cst}{2} \frac{x}{M} 
\frac{\rho(n_1)}{n_1}
A(n_1 n_2) + O(\frac{\rho(n_1)}{n_1} (\log x + x M^{-2} (n_1 n_2
M)^{\pwr+\epsilon})).
\end{align*}
The last line follows from writing
\begin{eqnarray*}
\sum_{l < x/M} \log\left(\textrm{min}\left(2,\frac{x}{lM}\right)\right) 
& = & 
\sum_{l \leq x/2M} \log 2 + \sum_{x/2M}^{x/M} \log \frac{x}{lM} \\
& = &
\frac{x}{2M} \log 2 + O(1) + \sum_{x/2M}^{x/M} \log \frac{x}{lM}
\end{eqnarray*}
and using the partial summation on the second sum with the summatory function
$L(t) = \sum_{x/2M < l < t} 1$:
\begin{eqnarray*}
\int_{x/2M}^{x/M} \log \frac{x}{tM} \ dL(t) 
& = & L(t) \log \frac{x}{tM} \Big|_{x/2M}^{x/M} +
\int_{x/2M}^{x/M} (t - \frac{x}{2M} + O(1)) \frac{1}{t} \ dt\\
& = &
0 + \int_{x/2M}^{x/M} (t - \frac{x}{2M} + O(1)) \frac{1}{t} \ dt\\
&& \frac{x}{2M} - \frac{x}{2M} \log 2 + O(\log x).
\end{eqnarray*}
Collecting our results, we have
\begin{eqnarray*}
V(x; M,N) 
& = &
\frac{\Cst}{2} \frac{x}{M} 
\sum_{n_1,n_2 < N} b_{n_1} b_{n_2} \frac{\rho(n_1)\rho(n_2)}{n_1 n_2}
A(n_1 n_2) \\
&& +\ O\left\{\sum_{n_1,n_2 < N} \left(\frac{\rho(n_2)}{n_2} x^{\epsilon} +
\frac{\rho(n_1)\rho(n_2)}{n_1 n_2} (\log x + x M^{-2} (n_1 n_2
M)^{\pwr+\epsilon}\right)\right\}.
\end{eqnarray*}
And since $\rho(n) \ll n^{\epsilon},$ 
\begin{align}
V(x; M,N) & = 
\frac{\Cst}{2} \frac{x}{M} 
\sum_{n_1,n_2 < N} b_{n_1} b_{n_2} \frac{\rho(n_1)\rho(n_2)}{n_1 n_2}
A(n_1 n_2) \nonumber \\
& \qquad +\ O(x M^{-2} (N^2 M)^{\pwr+\epsilon}).\label{eV}
\end{align}
\subsection{Estimation of $W(x; M,N)$}
Let the symbol $\Phi_{n_1,n_2}$ denote the set of quadruples $(m,v,k_1,k_2)$ such that 
\begin{align*}
M < m < 2M,\  & (m,n_1 n_2) = 1,\\
0 \leq v < m,\  & G(v) \equiv 0 \mod m,\\
k_1,k_2 < x,\ & k_1 \equiv k_2 \equiv v \mod m,\\
G(k_1) \equiv 0 \mod {n_1},\  &\ G(k_2) \equiv 0 \mod {n_2};
\end{align*}
and let 
\begin{align*}
T(n_1,n_2; x, M) = \sum_{\Phi_{n_1,n_2}} 1.
\end{align*}
Then
\begin{eqnarray*}
W(x; M,N)  =  \sum_{n_1, n_2 < N} b_{n_1} b_{n_2} T(n_1,n_2; x,M).
\end{eqnarray*}
Proceeding in a similar fashion as the last section, we write 
\begin{eqnarray*}
k_1 = m l_1 + v,& k_2 = m l_2 + v,
\end{eqnarray*}
where $l_1,l_2 < x/M.$  Then we may re-parameterize the sum $T(n_1,n_2; x,M)$ in
terms of the variables $m,v,l_1,$ and $l_2.$  As $l_1,l_2 < x/M,$ we may
eliminate the conditions on  $k_1$ and $k_2$ with $m < (x-v)/l_1,\ (x-v)/l_2.$  Thus we have equivalent conditions
\begin{align}
M < m < \min\left(2M, \frac{x-v}{l_1}, \frac{x-v}{l_2}\right),\ &  (m,n_1 n_2) = 1 \label{eqn_replace} \\
0 \leq v < m,\ & G(v) \equiv 0 \mod m,\nonumber \\
m < \frac{x-v}{l_1},\ & m< \frac{x-v}{l_2} \nonumber \\
G(m l_1 + v) \equiv 0 \mod {n_1},\ & G(m l_2 + v) \equiv 0 \mod{n_2}.\label{l1_l2_cond}
\end{align}
Again, we note that $0\leq v < m$ so replacing $x-v$ by $x$ in
\eqref{eqn_replace} results in an error
\begin{align*}
2 \sum_{l_1  \leq l_2 < x/M}
	\sum_{(x-2M)/l_2 < m < x/l_2}
	\sum_{\substack{0 \leq v < m \\ G(v) \equiv 0 \mod m}}
	\sum_{\substack{n_1 \mid G(ml_1+v) \\ n_2 \mid G(ml_2 + v)}} 1
\end{align*}
which is
\begin{align*}
\ll
\sum_{l_1  \leq l_2 < x/M}
\sum_{(x-2M)l_2 < m < x/l_2}
	\rho(m) 
\ll 
x^{1+\epsilon/2} \sum_{l_2 < x/M} \frac{M}{l_2} \ll x^{1+\epsilon}.
\end{align*}
Denote the modified sum by $T^*(n_1,n_2; x,M).$  In an attempt to combine the
conditions \eqref{l1_l2_cond}, let $c$ be the solution of the
system of congruences
\begin{eqnarray*}
c \equiv l_1\  \bigg(\textrm{mod}\ {\frac{n_1}{(n_1,n_2)}}\bigg), & c \equiv
l_2\ \bigg(\textrm{mod}\ {\displaystyle \frac{n_2}{(n_1,n_2)}}\bigg),\\
c \equiv l_1\  \bigg(\textrm{mod}\ {(n_1,n_2)}\bigg), & 0\leq c < [n_1,n_2].
\end{eqnarray*}
Such a solution exists and is unique by the Chinese remainder theorem.  Let
$\Omega = cm + v.$  Then 
\begin{equation}
m \leq \Omega < (c+1)m\  \textrm{and}\  G(\Omega) \equiv 0 \mod {m[n_1,n_2]} \label{eq_4}
\end{equation}
the latter following from \eqref{l1_l2_cond}.  

If the congruence conditions in
\eqref{l1_l2_cond} are satisfied then by construction of $G(n),$ we must have
that both $n_1$ and $n_2$ are odd. Consequently, $d = (n_1,n_2)$ is odd and we can deduce from \eqref{l1_l2_cond} that
\begin{eqnarray}
v \equiv -m \frac{l_1 + l_2}{2} \mod d \qquad \textrm{and} \qquad
G\bigg(m\cdot \frac{l_1-l_2}{2}\bigg)
\equiv 0 \mod d.\label{eq_3}
\end{eqnarray}
Let $\mu$ be the the reduced residue class of $m$ modulo $d,$ and let $\omega =
\left(c - \frac{l_1+l_2}{2}\right)\mu.$  Then we see that
$\omega  =  \mu (l_1 - l_2)/2  \equiv cm +v \equiv \Omega \mod d$ and 
so $G(\mu (l_1 - l_2)/2) \equiv 0 \mod d.$
Then \eqref{eq_4}
and \eqref{eq_3} give
\begin{eqnarray}
G\bigg(\mu \frac{l_1 - l_2}{2}\bigg) \equiv 0 \mod d & \textrm{and} & \Omega \equiv \omega
\mod d.
\end{eqnarray}
Using the above substitutions, we have
\begin{eqnarray*}
T^{*}(n_1,n_2; x, M) = 
\sum_{\substack{l_1,l_ < x/M \\
				l_1 \equiv l_2 \mod {(2,n_1,n_2)}}}
\sum_{\mu: G(\mu) \equiv 0 \mod d}
\sum_{\substack{M < m < \min (2M, x/l_1, x/l_2),\ 
(m,n_1 n_2) = 1\\
m \equiv \mu \mod d,\ G(\Omega) \equiv 0 \mod {m[n_1,n_2]}}} 1.
\end{eqnarray*}
The inner most sum is $P(M_1, M; q, d, \mu, \omega, \alpha, \beta).$  Here 
\begin{align*}
M_1 &= \min \left(2M, \frac{x}{l_1},\frac{x}{l_2}\right)\\
q &= [n_1,n_2]\\
\alpha &= c/q,\ \beta = (c+1)/q.
\end{align*}
Thus, by Proposition \ref{lemma:4}, we have
\begin{align*}
T^{*}(n_1,n_2; x, M) 
& = 
\sum_{\substack{l_1,l_2 < x/M \\
				l_1 \equiv l_2 \mod {(2,n_1,n_2)}}}
\sum_{\mu: G(\mu) \equiv 0 \mod d}
\bigg\{\Cst \frac{\rho(q/d)}{q} \frac{A(q)}{\phi(d)} \\
& \qquad \cdot \left( \min \bigg\{2M, \frac{x}{l_1},\frac{x}{l_2}\bigg\}- M\right) +  O((qM)^{\pwr+\epsilon})\bigg\}\\
& =  \sum_{\substack{l_1,l_2 < x/M \\
				l_1 \equiv l_2 \mod {(2,n_1,n_2)}}}
\bigg\{\Cst \frac{\rho(q)}{q} \frac{A(q)}{\phi(d)}\left(
\min \left(2M, \frac{x}{l_1},\frac{x}{l_2}\right)- M\right)\\
&\qquad +  O((qM)^{\pwr+\epsilon})\bigg\} 
\end{align*}
Substituting $q,d$ we see this is 
\begin{eqnarray*}
\frac{2 \cdot \Cst \cdot \rho([n_1,n_2])}{n_1,n_2}  \frac{A(n_1 n_2)}{\phi((n_1,n_2))}
\sum_{\substack{l_1 < l_2 < x/M \\
				l_1 \equiv l_2 \mod {(2,n_1,n_2)}}}
\phi((n_1,n_2,l_1-l_2))\left(\min \left(2M, \frac{x}{l_2}\right)-
M\right)\\
 + O\left(\frac{(n_1 n_2)}{n_1 n_2} \rho(n_1 n_2) x\right) +
 O\left(\left(\frac{x}{M}\right)^2(n_1 n_2 M)^{\pwr+\epsilon}\right);
\end{eqnarray*}
and here, the first error term comes from the terms in the sum where $l_1 =
l_2.$  Also note that $A([n_1,n_2]) = A(n_1 n_2)$ as $\phi(q)/q = \phi(qy)/qy$ and $(2,q) = (2,qy)$ for any $y$ dividing $q.$  

For the sake of using Dirichlet's hyperbola method, we write $\phi = 1* \psi$
for some multiplicative function $\psi.$  It is easy to see that $\psi(p) = p-2$
(which is all we will need as we will be summing over squarefree numbers).
Then 
\begin{align*}
\sum_{\substack{0 < l_1 < l_2  \\
					l_1 \equiv l_2 \mod {(2,n_1,n_2))}}} \phi((n_1,n_2,l_1 - l_2))
& = 
\sum_{t \mid (n_1,n_2)} \psi(t) 
\sum_{\substack{0 < l_1 < l_2 \\
					l_1 \equiv l_2 \mod {t (2,n_1,n_2))}}} 1 \\
& = 
\sum_{t \mid (n_1,n_2)} \psi(t) 
\left(\frac{l_2}{t (2,n_1,n_2)} + O(1)\right) \\
& = 
\frac{l_2}{(2,n_1,n_2)}\sum_{t \mid (n_1,n_2)} \psi(t)/t + O(\phi((n_1,n_2))).
\end{align*}
From the product expansion of $\sum_{t|(n_1,n_2)} \psi(t)/t,$ we obtain
\begin{align*}
\sum_{\substack{0 < l_1 < l_2  \\
					l_1 \equiv l_2 \mod {(2,n_1,n_2))}}} \phi((n_1,n_2,l_1 - l_2))
&= \frac{l_2}{(2,n_1,n_2)}\prod_{p \mid (n_1,n_2)}2\cdot\frac{p-1}{p} + O(\phi((n_1,n_2)))\\
& = 
l_2 \frac{\tau((n_1,n_2))}{(2,n_1,n_2)}\cdot \frac{\phi((n_1,n_2))}{(n_1,n_2)}
+ O(\phi((n_1,n_2))).
\end{align*}
Also
\begin{eqnarray*}
\sum_{l_2 < x/M} l_2\bigg(\min\bigg\{2M,\frac{x}{l_2}\bigg\}-M\bigg) 
 = \frac{x^2}{4M} + O(x),
\end{eqnarray*}
and
\begin{eqnarray*}
\rho([n_1,n_2])\frac{d((n_1,n_2))}{(2,n_1,n_2)}  = \rho(n_1) \rho(n_2).
\end{eqnarray*}
Combining these relations, 
\begin{eqnarray*}
T^*(n_1,n_2; x,M) & = & \frac{\Cst x^2}{2M}
\frac{\rho(n_1)\rho(n_2)}{n_1 n_2} A(n_1 n_2) \\
&& + O\left(\frac{\rho(n_1 n_2)}{n_1 n_2} (n_1,n_2) x +
\left(\frac{x}{M}\right)^2(n_1 n_2 M)^{\pwr+\epsilon}\right).
\end{eqnarray*}
As the error between $T(n_1,n_2; x,M)$ and $T^*(n_1,n_2; x,M)$ is $\ll
x^{1+\epsilon},$  we have
\begin{align}
W(x; M,N) &=  
 \frac{\Cst x^2}{2M}
\sum_{n_1,n_2 < N} b_{n_1} b_{n_2} \frac{\rho(n_1)\rho(n_2)}{n_1 n_2} A(n_1 n_2)
\nonumber \\
& \qquad + \sum_{n_1,n_2<N}O\bigg( \frac{\rho(n_1 n_2)}{n_1 n_2} (n_1,n_2) x +
\left(\frac{x}{M}\right)^2 (n_1 n_2 M)^{\pwr+\epsilon}\bigg) +
O(x^{1+\epsilon})\nonumber \\
& = 
\frac{\Cst x^2}{2M}
\sum_{n_1,n_2 < N}b_{n_1} b_{n_2}\frac{\rho(n_1)\rho(n_2)}{n_1 n_2} A(n_1 n_2)
\nonumber \\
& \qquad + O\bigg(x^{1+\epsilon} +
\left(\frac{x}{M}\right)^2 N^2 (N^2 M)^{\pwr+\epsilon}\bigg).\label{eW} 
\end{align}
This estimate for $W(x; M,N)$ is sufficient.

Combining our three estimates for  $W,U,$ and $V$ in
\eqref{eW}, \eqref{eU},and \eqref{eV}, respectively, we see that
\begin{align*}
\mathcal{M}(x; M,N) &=
W(x; M,N) - 2x V(x; M,N) + x^2 U(M,N) \\
 &=
\frac{x^2 \Cst}{2M}\left(1 -2 + 1\right)
\sum_{n_1,n_2 < N} b_{n_1} b_{n_2} \frac{\rho(n_1)\rho(n_2)}{n_1 n_2} A(n_1 n_2)
\\ &\qquad + O(x + N^{15/4} M^{-9/4} x^2)x^{\epsilon}
\end{align*}
with the implied constant depending on $\epsilon.$  This completes the proof of Proposition \ref{proposition_1}.

%% file: chapter5.tex
\setlength{\parskip}{0.5cm}
\chapter{The Estimation of the Sifting Functions}
We will make use of the following proposition to find a lower bound for
$W(\mathscr{A},z)$ in the next section.
\begin{proposition}\label{est_sifting}
Let $y = x^{16/15},\ 0 < \gamma < 1/2,\ z = x^{\gamma},\ z \leq z_q <
x^{1/2}$ and $0\leq c_q \leq 1.$  Then for any $\epsilon > 0$
\begin{eqnarray*}
\sum_{\substack{q < x^{1-\epsilon} \\ (q,P(z_q)) = 1}} 
		c_q S(\mathscr{A}_q; z_q)
<
V(z)x\left\{
		\sum_{\substack{q < x^{1-\epsilon} \\ (q,P(z_q)) = 1}} 
		c_q \frac{\rho(q)}{q} F\left(\frac{\log (y/q)}{\log z_q}\right)
		\frac{\log z}{\log z_q} + O_{\gamma}(\epsilon)
		\right\}
\end{eqnarray*}
for $x>x_0(\epsilon,\gamma).$  
\end{proposition}
\begin{proof} For real numbers $Q,Z$, let $H(Q,Z)$ denote the set of integers
$q$ that satisfy the conditions
\begin{eqnarray*}
Q \leq 2Q,\ Z \leq z_q < 2Z,& (q,P(z_q)) = 1.
\end{eqnarray*}
Now for each $q$ we apply Lemma \ref{lemma_2} with $M = x^{1-\epsilon}/Q$ and 
$N = x^{1/15 - \epsilon}$.  It follows that for any $\nu > 0,$
\begin{eqnarray*}
S(\mathscr{A}_q,z_q) & \leq &
V(Z)\frac{\rho(q)}{q}x
\bigg\{F\bigg(\frac{\log MN}{Z}\bigg) + O_{\gamma}(\nu)\bigg\} \\
&& \ + 2^{\nu^{-10}} \sum_{m < M,\ m| P(Z)} \bigg| 
	\sum_{n<N,\ n,m)=1} b_n r(\mathscr{A}, qmn)\bigg|.
\end{eqnarray*}
Note that $r(\mathscr{B},d) = r(\mathscr{A}_q,d) = r(\mathscr{A},qd).$
Multiplying by $c_q$ and summing over $q$ we see that
\begin{align*}
\sum_{H(Q,Z)} c_q S(\mathscr{A}_q,z_q) 
 & <  \bigg( \sum_{q \in H(Q,Z)} c_q \frac{\rho(q)}{q}  \bigg) V(Z) x \bigg\{
	F\left(\frac{\log (y/Qx^{2\epsilon}}{\log Z}\right) 
	+O(\nu)\bigg\} + O_{\epsilon}(2^{\nu^{-10} x^{1-\epsilon}}),
\end{align*}
where the error term comes from the estimate in Corollary \ref{corollary1}.
But our assumption on the sifting density of $\rho$ we have 
\begin{eqnarray*}
V(Z) = V(z)\frac{\log z}{\log z_q}\left(1 + O\left(\frac{1}{\log
z}\right)\right),
\end{eqnarray*}
and 
\begin{eqnarray*}
F\left(\frac{\log(y/Qx^2\epsilon)}{\log Z}\right)  = 
F\left(\frac{\log (y/q)}{\log z_q}\right) + O_{\gamma}(\epsilon).
\end{eqnarray*}
Since the number of classes $H(Q,Z)$ needed to cover all possibilities of $q$
(namely $1 \leq q < x^{1-\epsilon}$) is $\ll \log^2 x,$ we obtain 
\begin{eqnarray*}
\sum_{\substack{q < x^{1-\epsilon} \\ (q,P(z_q)) = 1}} 
		c_q S(\mathscr{A}_q; z_q)
& < &
V(z)x\Big\{
		\sum_{\substack{q < x^{1-\epsilon} \\ (q,P(z_q)) = 1}} 
		c_q \frac{\rho(q)}{q} F\left(\frac{\log (y/q)}{\log z_q}\right)
		\frac{\log z}{\log z_q} + O_{\gamma}(\nu + \epsilon) \Big\}\\
&& + O_{\epsilon}\bigg(\log^2 x \cdot x^{1-\epsilon} 2^{\nu^{-10}}\bigg)
		\\
& \leq &
		\sum_{\substack{q < x^{1-\epsilon} \\ (q,P(z_q)) = 1}} 
		c_q \frac{\rho(q)}{q} F\left(\frac{\log (y/q)}{\log z_q}\right)
		\frac{\log z}{\log z_q} + O_{\gamma}(\epsilon).
\end{eqnarray*}
\end{proof}

%% file: chapter6.tex
\setlength{\parskip}{0.5cm}
\newcommand{\eless}{\stackrel{\scriptscriptstyle \infty}{<}}
\newcommand{\egreat}{\stackrel{\scriptscriptstyle \infty}{>}}
\chapter{Estimation of $W(\mathscr{A},z)$}
Recall from equation \eqref{eqn_W2}
\begin{align*}
W(\mathscr{A},z) 
&=
S(\mathscr{A},z) +
\frac{1}{3 - \lambda}\bigg\{\sum_{z \leq p < x^{1/2}} 
\sum_{z \leq p_1 < p} \frac{\log p/p_1}{\log x} S(\mathscr{A}_{p p_1},p_1) \\
&  -
\sum_{z \leq p < x^{1/2}}\bigg(\left(1-\frac{2\log p}{\log x}\right) 
	S(\mathscr{A}_p,p) + \frac{\log p}{\log x} S(\mathscr{A}_p,z)\bigg) \\
&  -
\sum_{x^{1/2}\leq p < x}\left(1 - \frac{\log p}{\log x} \right)
		S(\mathscr{A}_p,z)\bigg\}.
\end{align*}
We can use Proposition \ref{est_sifting} to estimate each of the sums except the
last.  For the last sum, we consider $\sum_{x^{1-\epsilon}\leq p < x}(1- \log p / \log x)
S(\mathscr{A}_p,z)$ and $\sum_{x^{1/2} \leq p < x^{1-\epsilon}}
(1 - \log p/\log x)S(\mathscr{A}_p,z)$ separately, applying Proposition
\ref{est_sifting} to the later.  For the former we crudely bound
\begin{align*}
\sum_{x^{1-\epsilon}\leq p < x}\left(1 - \frac{\log p}{\log x} \right)
		S(\mathscr{A}_p,z) &\ll 
\sum_{x^{1/2}\leq p < x}\left(1 + \frac{\log p}{\log x} \right)
		\frac{x}{p\log(x/p)} \\
&  \ll
\frac{x}{\log{x}}\sum_{x^{1/2}\leq p < x} (\frac{1}{p} + 
\frac{\log p}{p \log x})  \\
& \ll x/\log x.
\end{align*}

Upon application of Proposition \ref{est_sifting}, we will use partial
summation; and for such a task, we require to know more about the summatory
function $P(t) = \sum_{\leq p < t} \rho(p)/p.$  
\begin{lemma} \label{nagellemma}
We have
\begin{align*}
P(t) = \log \log t + b + o_G(1),
\end{align*}
for some constant $b.$
\end{lemma}
\begin{proof}
From
a result of Nagel in \cite{Nagel}, we have
\begin{align}
L(t) = \sum_{p < t} \frac{\rho(p) \log p}{p} = \log t + O_G(1) \qquad (t\geq 2).
\label{eqn_L}
\end{align}
Let $R(t) = L(t) - \log t = O_G(1).$  Then,
\begin{align*}
P(t) & =
\int_{2^{-}}^{t} \frac{1}{\log x} \ dL(x)  
= \frac{L(t)}{ \log t} + \int_{2}^{t} \frac{L(x)}{x \log^2 x} \ dx\\
&=  \bigg\{1+ \frac{R(t)}{ \log t}\bigg\} + \bigg\{\int_{2}^{t} \frac{1}{x \log x} \
dx+\int_{2}^{t} \frac{R(x)}{x \log^2 x} \ dx \bigg\}\\
&=  \log \log t + \bigg\{1 - \log \log 2\ + \int_{2}^{\infty}\frac{R(x)}{x \log^2 x}
\ dx \bigg\}  + \bigg\{\frac{R(t)}{ \log t} - \int_{t}^{\infty} \frac{R(x)}{x \log x} \ dx\bigg\}\\
&= \log \log t + b + o_G(1),
\end{align*}
where $b$ is the constant $1 -\log\log 2  + \int_2^{\infty} R(x)/x \log^2 x \
dx.$ 
\end{proof}
Choose $\gamma=1/5.$  For the remainder of this section we will use the notation
$f(x) \eless g(x)$ to mean $f(x)$ is `eventually less than' $g(x)$ (ie $f(x)<g(x)$ for
$x>x_0(G,\epsilon, \gamma)$). Similarily the notation $f(x) \egreat g(x)$ will
mean $f(x) > g(x)$ for $x>x_0(G,\epsilon,\gamma).$  

From Proposition \ref{est_sifting}, it follows that
\begin{align}
\sum_{z \leq p < x^{1/2}}\left(1-2\frac{\log p}{\log x}\right) 
	S(\mathscr{A}_p,p)
& \eless V(z) x \bigg\{\sum_{z \leq p < x^{1/2}}\left(1-2\frac{\log p}{\log x}\right) 
		\frac{\rho(p)}{p} \nonumber \\ &\qquad \cdot F\bigg(\frac{\log y/p}{\log p}\bigg) \frac{\log z}{\log
		p}+O_{\gamma}(\epsilon)\bigg\}.\label{eqn200}
\end{align}
Upon application of partial summation, the sum becomes
\begin{align} \label{eqn100}
\int_{z}^{x^{1/2}}
	\left(1-2\frac{\log v}{\log x}\right) 
 	F\bigg(\frac{\log y/v}{\log v}\bigg) \frac{\log z}{\log v} \ dP(v)
= 
\int_{\gamma}^{1/2}
	(1-2u)
 	F\bigg(\frac{\alpha - u}{u}\bigg) \frac{\gamma}{u} \ dP(x^u).
\end{align}
To evaulate this integral, we make use of the lemma:
\begin{lemma} \label{intgr}
Suppose $A(t)$ is a differientiable function with bounded derivative on
$[\alpha,\beta],$ and $B(t) = b(t) + o(1),$ where $o(1)\rightarrow 0$ as some
parameter $x\rightarrow\infty.$  Then
\begin{align*}
\int_{\alpha}^{\beta} A(t) dB(t) =
\int_{\alpha}^{\beta} A(t) db(t) + o(1).
\end{align*}
\end{lemma}
\begin{proof}
We have
\begin{align*}
\int_{\alpha}^{\beta} A(t)\ dB(t) & = 
\int_{\alpha}^{\beta} A(t)\ db(t) + \int_{\alpha}^{\beta} A(t)\ d\{B(t) - b(t)\}
\end{align*}
By partial summation the second integral is
\begin{align}
\int_{\alpha}^{\beta} A(t) d\{B(t) - b(t)\} & =
\{B(t)-b(t)\}A(t) \big|_{\alpha}^{\beta} - \int_{\alpha}^{\beta}  (B(t)-b(t))
A'(t) \ dt \label{eqn101}.
\end{align}
Since $A'(t)$ is bounded on $[\alpha,\beta]$, so is $A(t);$ thus \eqref{eqn101}
is $o(1).$ 
\end{proof}
As $P(x^u) = \log \log x^u + b + o_G(1),$ Lemma \ref{intgr} gives
\eqref{eqn100} is 
\begin{align*}
  \int_{\gamma}^{1/2}
	(1-2u)
 	F\bigg(\frac{\alpha - u}{u}\bigg) \frac{\gamma}{u} \ \frac{du}{u} + o_G(1),
\end{align*}
and so \eqref{eqn200}  becomes
\begin{align*}
\sum_{z \leq p < x^{1/2}}\left(1-2\frac{\log p}{\log x}\right) 
	S(\mathscr{A}_p,p)
& \eless V(z) x \bigg\{
  \int_{\gamma}^{1/2}
	(1-2u)
 	F\bigg(\frac{\alpha - u}{u}\bigg) \frac{\gamma}{u} \ \frac{du}{u}
	+O_{\gamma}(\epsilon)\bigg\}.
\end{align*}
Similarily,
\begin{align*}
\sum_{z \leq p < x^{1/2}} \frac{\log p}{\log x} S(\mathscr{A}_p,z)
& \eless 
V(z) x \bigg\{\sum_{z \leq p < x^{1/2}} \frac{\log p}{\log x} \frac{\rho(p)}{p}
F\bigg(\frac{\log y/p}{\log z}\bigg)+O_{\gamma}(\epsilon)\bigg\}. \\
& = V(z) x \bigg\{\int_{\gamma}^{1/2} u F\bigg(\frac{\alpha - u}{\gamma}\bigg) \frac{du}{u} + O_{\gamma}(\epsilon)\bigg\},
\end{align*}
and 
\begin{align*}
\sum_{x^{1/2}\leq p < x}\left(1 - \frac{\log p}{\log x}\right) 
		S(\mathscr{A}_p,z)
& \eless
V(z) x \bigg\{\sum_{z \leq p < x^{1/2}} \bigg(1 - \frac{\log p}{\log x}\bigg)
\frac{\rho(p)}{p} F\bigg(\frac{\log y/p}{\log z}\bigg) + O_{\gamma}(\epsilon)\bigg\}.\\
& = V(z) x \bigg\{\int_{\gamma}^{1/2} (1-u) F\bigg(\frac{\alpha -
u}{\gamma}\bigg)
\frac{du}{u}+ O_{\gamma}(\epsilon)\bigg\}.
\end{align*}
And lastly the double sum is 
\begin{align*}
\sum_{z \leq p < x^{1/2}} 
\sum_{z \leq p_1 < p} \frac{\log p/p_1}{\log x} S(\mathscr{A}_{p p_1},p_1) 
& \egreat
			V(z)x\bigg\{\int_{\gamma}^{1/2} \int_{\gamma}^{t} (u-t) \frac{\gamma}{t}
					f\bigg(\frac{\alpha - u - t}{t}\bigg) 
					\frac{du}{u} \frac{dt}{t}+ O_{\gamma}(\epsilon)\bigg\}.
\end{align*}
And since $\frac{1}{3 - \lambda} \rightarrow 1$ as $x\rightarrow \infty,$ it
follows that 
\begin{align*}
W(\mathscr{A},z) &\egreat V(z)x \bigg\{f\bigg(\frac{\alpha}{\gamma}\bigg)
			+ \int_{\gamma}^{1/2} \int_{\gamma}^{t} (u-t) \frac{\gamma}{t}
					f\bigg(\frac{\alpha - u - t}{t}\bigg) 
					\frac{du}{u} \frac{dt}{t} \\
	& \qquad - \int_{\gamma}^{1/2}\bigg[(1-2u)\frac{\gamma}{u} 
			F\bigg(\frac{\alpha - u}{u}\bigg) + uF\bigg(\frac{\alpha-u}{\gamma}\bigg)
			\bigg] \frac{du}{u}\\
	& \qquad - \int_{1/2}^1(1-u)F\bigg(\frac{\alpha - u}{\gamma}\bigg)\frac{du}{u} 
	-\epsilon\}\\
	&= V(z) x \{W - \epsilon\},	
\end{align*}
with $W$ being implicitly defined, $\alpha=16/15,$ and $\epsilon > 0.$ 
From Mertens prime number theorem, it follows that
\begin{align*}
V(z) \sim \Gamma_G e^{-C} (\log z)^{-1} = \Gamma_G e^{-C}(\gamma \log x)^{-1}.
\end{align*}
The functions $F(s)$ and $f(s)$ are elementary in the intervals 
$0<s\leq 3$ and $0 < s \leq 4$ respectively.  But as we will choose
$\gamma=1/5$, we require $F(s)$ and $f(s)$ outside these ranges.  
From the differential-difference equations for $F(s)$ and $f(s)$ in Section
\ref{ErrorTerm}, it follows that
\begin{align*}
s F(s) &=2e^{C} \bigg\{1+ \int_{2}^{s-1} \log(u-1) \frac{du}{u}\bigg\} 
\qquad \textrm{if } 3 \leq s \leq 5,\\
sf(s) &= 2 e^{C}\bigg\{ \log (s-1) + \int_{3}^{s-1} \int_{2}^{t-1} \log(u-1) 
\frac{du}{u}\frac{dt}{t} \bigg\} \qquad \textrm{if }4 \leq s \leq 6.
\end{align*}
Using these formulae, we obtain
\begin{align*}
W &= 2e^C\gamma\bigg\{ \log (\alpha - \gamma) - \frac{\alpha - 1 }{\alpha} \log
(\alpha -1) \\
&\qquad - \int_{2}^{\alpha/\gamma - 2} \bigg[ 
t \log \frac{\alpha(t+1)}{(\alpha - \gamma)(t+2)} + 
\log \big(1 - \frac{\gamma t}{\alpha - \gamma} \big)(t+1)\bigg] 
\frac{\log(t-1)}{t(t+1)} \ dt\bigg\}.
\end{align*}
With the help of Maple, we see that
\begin{align*}
W > 2e^C\gamma \cdot (.014057...) > 2e^C\gamma /154,
\end{align*}
the latter obtained by Iwaniec by considering an integral more suitable for
manual calculations. It follows that for sufficiently large $x$
\begin{align*}
W(\mathscr{A},z) > \frac{\Gamma_G}{77} \frac{x}{\log x}
\end{align*}
as required.